\documentclass[10pt,a4paper]{amsart}
\usepackage[utf8]{inputenc}
\usepackage[T1]{fontenc}
\usepackage[french, english]{babel}
\usepackage{amsmath,amsfonts,amssymb,amsthm}
\usepackage{stmaryrd} 
\usepackage{array}
\usepackage{enumerate}
\usepackage{hyperref}
\usepackage{mathabx}
\newcommand{\Q}{\mathbb{Q}}
\newcommand{\N}{\mathbb{N}}
\newcommand{\Z}{\mathbb{Z}}
\newcommand{\R}{\mathbb{R}}
\newcommand{\F}{\mathbb{F}}
\newcommand{\C}{\mathbb{C}}
\newcommand{\Gal}{\mathrm{Gal}}

\newcommand{\tors}{\mathrm{tors}}
\newcommand{\sat}{\mathrm{sat}}
\newcommand{\End}{\mathrm{End}}

\newcommand{\G}{\mathbb{G}}
\newcommand{\poubelle}[1]{}

\theoremstyle{plain}
\newtheorem{thm}{Théorème}[section]
\newtheorem{prop}[thm]{Proposition}
\newtheorem{cor}[thm]{Corollaire}
\newtheorem{lmm}[thm]{Lemme}

\newtheorem{conj}[thm]{Conjecture}

\newtheorem*{lmm*}{Lemme}
\newtheorem*{conj*}{Conjecture}
\newtheorem*{thm*}{Théorème}
\newtheorem*{claim*}{Fait}

\theoremstyle{plain}
\newtheorem{rqu}[thm]{Remarque}

\theoremstyle{plain}
\newtheorem{defn}[thm]{Définition}

\usepackage[all]{xy}
\usepackage{tikz}
\usetikzlibrary{shadows,decorations.pathmorphing}
\usepackage{comment}
\usepackage{ifthen}

\begin{document}

\title[Points de petite hauteur sur $\G_m^n \times A$]{Points de petite hauteur sur une variété semi-abélienne de la forme $\G_m^n \times A$}
\date{}
\author{Arnaud Plessis}
\address{Arnaud Plessis : Academy of Mathematics ans Systems Science, Morningside Center of Mathematics, Chinese Academy of Sciences, Beijing 100190}
\email{plessis@amss.ac.cn}

\maketitle 

\selectlanguage{french}

\begin{abstract}
Récemment, Rémond a énoncé une conjecture très générale pour la minoration d'une hauteur canonique sur une variété abélienne ou sur une puissance du groupe multiplicatif.

Dans ce papier, nous étendons un cas particulier de cette conjecture en considérant des variétés semi-abéliennes de la forme $\G_m^n \times A$.
Cela nous permet d'interpréter plusieurs résultats déjà présents dans la littérature comme des cas particuliers de cette nouvelle conjecture. 
Enfin, nous donnons un nouvel exemple allant dans le sens de celle-ci. 
\end{abstract}

\selectlanguage{english}

\begin{abstract}
Recently, Rémond stated a very general conjecture on lower bounds of a canonical height on either an abelian variety or a power of the multiplicative group. 

In this note, we extend a particular case of this conjecture to semi-abelian varieties of the form $\G_m^n \times A$. 
This allows us to connect many results already existing in the litterature. 
Finally, we give new examples for which this conjecture holds.  
\end{abstract}

\selectlanguage{french}

\section{Introduction.}
Les notations standards ci-dessous seront valables pour tout ce texte.
\begin{itemize}
\item $\overline{k}$ : clôture algébrique d'un corps commutatif $k$;
\item $\mathcal{O}_k$ : anneau des entiers d' un corps de nombres ou d'un corps local $k$;
\item $\G_m^n$ :  groupe multiplicatif de dimension $n\geq 0$;
\item $\zeta_n$ : racine de l'unité d'ordre un entier $n\geq 1$;
\item $\langle \alpha \rangle=\{\alpha^n, n \in\Z\}$ le groupe engendré par $\alpha\in \overline{\Q}^*$; 
\item $\langle \alpha\rangle_{\sat, p}=\{g\in \overline{\Q}^* \vert \; \exists n\geq 0, g^{p^n}\in \langle \alpha\rangle\}$ avec $\alpha\in \overline{\Q}^*$ et $p\geq 2$ un nombre premier;  
\item  $h : \overline{\Q}^* \to \R$ la hauteur (logarithmique, absolue) de Weil; 
\item $\hat{h}_{\G_m^n}(\boldsymbol{\alpha})=h(\alpha_1)+\dots +h(\alpha_n)$, où $\boldsymbol{\alpha}=(\alpha_1,\dots,\alpha_n)\in\G_m^n(\overline{\Q})$;
\item $\hat{h}_A:=\hat{h}_{A,\mathcal{L}} : A(\overline{k}) \to \R$ la hauteur de Néron-Tate d'une variété abélienne $A$ définie sur un corps de nombres $k$ associée à un fibré en droite ample et symétrique $\mathcal{L}$ (que l'on ne mentionnera plus à l'avenir);
\item Pour une variété semi-abélienne $\G$ définie sur un corps de nombres $k$ et un sous-groupe $\Gamma \subset \G(\overline{k})$, notons 
\begin{itemize}
\item $O_\G$ l'élément neutre de $\G$; 
\item $[m] : \G(\overline{k}) \to \G(\overline{k})$ la multiplication usuelle par un entier $m$; 
\item $k(F)$ le corps de rationalité d'un ensemble $F\subset \G(\overline{k})$, i.e. le plus petit sous-corps de $\overline{k}$ sur lequel tous les points de $F$ sont rationnels. 
\item le rang de $\Gamma$ comme étant la dimension du $\Q$-espace vectoriel $\Gamma \otimes_\Z \Q$;
\item $\Gamma_\sat =\{ P\in \G(\overline{k}) \mid \exists m\geq 1, [m]P\in \End (\G).\Gamma \}$ le groupe saturé de $\Gamma$, où $\End (\G).\Gamma$ désigne le groupe engendré par les éléments de la forme $\phi(\gamma)$ avec $\phi$ un endomorphisme de $\G$ et $\gamma\in\Gamma$; 
\item $\G_\tors=\{O_\G\}_\sat$ le groupe de torsion de $\G$ et $\mu_\infty= (\G_m)_\tors$.
\end{itemize}
\end{itemize}

La conjecture ci-dessous correspond au cas de la forme de degré un d'une conjecture récente et profonde de Rémond \cite[Conjecture 3.4]{R\'emond}.

\begin{conj} \label{conj 1}
Supposons que $\G=\G_m^n$ ou que $\G$ soit une variété abélienne définie sur un corps de nombres $k$. 
Si $\Gamma\subset \G(\overline{k})$ est un groupe de rang fini, alors il existe $c_\Gamma>0$ tel que $\hat{h}_\G(P)\geq c_\Gamma$ pour tout $P\in \G(k(\Gamma))\backslash \Gamma_\sat$.
\end{conj}

Il existe peu de résultats connus de la conjecture \ref{conj 1}, voir, entre autres, \cite{Baker, AmorosoZannier} et \cite[Theorem 2]{Habegger}. 
Par exemple, le cas d'apparence simple $k=\Q$, $\G=\G_m$ et $\Gamma= \langle 2 \rangle_\sat=\{ \zeta 2^x, \zeta\in \mu_\infty, x\in \Q\}$ reste encore inconnu, i.e. que l'on ne sait toujours pas s'il existe $c>0$ tel que $h(\alpha)\geq c$ pour tout $\alpha \in \Q(\mu_\infty, 2^{1/2}, 2^{1/3}, \dots)^*\backslash \langle 2 \rangle_\sat$. 
Amoroso a cependant prouvé que cela était vrai si $\alpha$ vivait dans le sous-ensemble $\Q(\zeta_3, 2^{ 1/3}, \zeta_{3^2}, 2^{1/3^2}, \dots)^*\backslash \langle 2 \rangle_\sat $.  
Plus précisément, il montra \cite[Theorem 3.3]{Amoroso}:

\begin{thm} \label{thm 1}
Soient $a\geq 2$ un entier et $p > 2$ un nombre premier tels que $p\nmid a$ et $p^2\nmid a^{p-1}-1$. 
Il existe alors $c>0$ tel que $h(\alpha) \geq c$ pour tout $\alpha\in \Q(\langle a \rangle_{\sat,p})^* \backslash \langle a \rangle_\sat$. 
\end{thm}
R\'emond a déjà montré que la conjecture \ref{conj 1} ne pouvait pas s'étendre aux variétés semi-abéliennes car celles-ci peuvent posséder des points de Ribet (voir section \ref{section 5} pour la d\'efinition) qui ne sont pas de torsion \cite[§ 5]{R\'emond}.

On sait que les sections de Ribet (en parlant grossi\`erement, une section de Ribet est une section \'etendant un point de Ribet, voir \cite{BMPZ} pour une d\'efinition formelle) sont les seules obstructions \`a la validit\'e de la conjecture de Manin-Munford relative pour des familles de surfaces semi-ab\'eliennes de dimension $1$ \cite{BMPZ}. 
La conjecture de R\'emond \'etant li\'ee \`a la conjecture de Zilber-Pink, nous pouvons alors s'attendre \`a ce que les points de Ribet soient les seules obstructions \`a la conjecture \ref{conj 1}. 

\begin{conj} \label{conj 2}
Soit $\G$ une variété semi-abélienne définie sur un corps de nombres $k$. 
Notons $\G_\mathrm{Rib}$ le groupe des points de Ribet.
Si $\Gamma\subset \G(\overline{k})$ est un groupe de rang fini, alors il existe $c_\Gamma>0$ tel que $\hat{h}_\G(P)\geq c_\Gamma$ pour tout $P\in \G(k(\Gamma))\backslash (\G_\mathrm{Rib} + \Gamma_\sat)$, o\`u $\G_\mathrm{Rib} + \Gamma_\sat$ d\'esigne le groupe engendr\'e par tous les points de Ribet et de $\Gamma_\sat$.
\end{conj}

Pour la construction classique de la hauteur $\hat{h}_\G$ dans le cas semi-ab\'elien, nous r\'ef\'erons le lecteur \`a \cite[section 2.3.c]{R\'emond}.

\`A partir de maintenant, nous \'etudions cette conjecture dans le cas plus abordable des variétés semi-abéliennes de la forme $\G_m^n \times A$, avec $A$ une variété abélienne. 
Nous montrerons dans la section \ref{section 5} que ses points de Ribet sont pr\'ecis\'ement ses points de torsion. 
Dans ce cas pr\'ecis, la conjecture \ref{conj 2} se r\'e\'ecrit alors comme suit :

\begin{conj}  \label{conj 3}
Soit $A$ une variété abélienne définie sur un corps de nombre $k$ et posons $\G=\G_m^n \times A$. 
Si $\Gamma \subset \G(\overline{k})$ est un groupe de rang fini, alors il existe $c_\Gamma>0$ tel que pour tout $(\boldsymbol{\alpha}, P) \in \G(k(\Gamma))\backslash \Gamma_\sat$, \[ \hat{h}_\G (\boldsymbol{\alpha}, P):= \hat{h}_{\G_m^n}(\boldsymbol{\alpha})+\hat{h}_A(P) \geq c_\Gamma.\]
\end{conj}

Posons $\mathbf{1}=(1,\dots,1)\in \G_m^n(\overline{k})$.
Si $\Gamma = \Gamma' \times \{O_A\}$, avec $\Gamma'\subset \G_m^n(\overline{k})$ un groupe de rang fini, la conjecture \ref{conj 3} prédit l'existence d'un $c>0$ tel que $\hat{h}_\G(\mathbf{1},P)\geq c$ pour tout $(\mathbf{1}, P)\in \G(k(\Gamma'))\backslash \Gamma_\sat$. 
Ainsi, $\hat{h}_A(P)\geq c$ pour tout $P\in A(k(\Gamma'))$ tel que $(\mathbf{1},P)\notin \Gamma_\sat$. 
Cette conjecture localise donc les points de petite hauteur de $A(k(\Gamma'))$, alors que la conjecture \ref{conj 1} les localise seulement pour $\G_m^n(k(\Gamma'))$.

De manière similaire, si $\Gamma = \{ \mathbf{1}\} \times \Gamma'$, avec $\Gamma' \subset A(\overline{k})$ un groupe de rang fini, la conjecture \ref{conj 3} donne des informations sur les points de petite hauteur dans $\G_m^n(k(\Gamma'))$, tandis que la conjecture \ref{conj 1} en donne uniquement pour $A(k(\Gamma'))$. 

Notons que tout sous-groupe de $\G_\tors=\{O_\G\}_\sat$ a un saturé égal à $\G_\tors$. 
Citons maintenant quelques résultats allant dans le sens de la conjecture \ref{conj 3}. 

 \begin{thm} [Amoroso-Zannier \cite{AmorosoZannier} ; Baker-Silverman \cite{BakerSilverman}] \label{thm 2}
 Il existe $c>0$ tel que $\hat{h}_\G(\boldsymbol{\alpha},P)\geq c$ pour tout $(\boldsymbol{\alpha}, P)\in \G(k(\mu_\infty))\backslash \G_\tors$, où $\G=\G_m \times A$.
 En d'autres mots, la conjecture \ref{conj 3} est satisfaite pour $\G= \G_m \times A$ et $\Gamma=\mu_\infty \times \{O_A\} \subset \G_\tors$.
  \end{thm}
  
\begin{thm} [Habegger, Corollary 1 \cite{Habegger}] \label{thm 3}
Soit $E$ une courbe elliptique définie sur $\Q$. 
Posons $\G=\G_m \times E$. 
Il existe alors $c>0$ tel que $\hat{h}_\G(\boldsymbol{\alpha}, P)\geq c$ pour tout $(\boldsymbol{\alpha}, P)\in \G(\Q(E_\tors))\backslash \G_\tors$.
En d'autres mots, la conjecture \ref{conj 3} est satisfaite pour $k=\Q$, $\G=\G_m \times E$ et $\Gamma=\{1\} \times E_\tors \subset \G_\tors$. 
\end{thm}

\begin{rqu} \label{rqu 1}
\rm{Dans le théorème \ref{thm 2}, \cite{AmorosoZannier} correspond à la \textit{partie torique}, i.e. qu'il n'y a pas de point de petite hauteur dans $k(\mu_\infty)^* \backslash \mu_\infty$ et \cite{BakerSilverman} à la \textit{partie abélienne}, i.e. qu'il n'y a pas de point de petite hauteur dans $A(k(\mu_\infty))\backslash A_\tors$.
La concaténation de ces deux théorèmes montre qu'il n'y a pas de point de petite hauteur dans \[(\G_m \times A)(k(\mu_\infty))\backslash (\mu_\infty \times A_\tors)= \G(k(\Gamma))\backslash \G_\tors,\] où $\G=\G_m \times A$ et $\Gamma= \mu_\infty \times \{O_A\}$. 
Cela correspond précisément au théorème \ref{thm 2}.

Pour montrer le théorème \ref{thm 3}, Habegger a procédé de manière similaire en montrant d'abord la partie torique \cite[Theorem 1]{Habegger}, puis la partie elliptique \cite[Theorem 2]{Habegger} qu'il a ensuite unifiées pour en déduire le théorème \ref{thm 3}. }
\end{rqu}

On se propose de donner dans cet article un premier exemple, à ma connaissance, de la conjecture \ref{conj 3} lorsque $\Gamma_\sat\neq \G_\tors$ et $A\neq O_A$  (lorsque $A=O_A$, un tel exemple a déjà été donné dans le théorème \ref{thm 1}).
C'est l'objet du corollaire \ref{cor 1}. 
Comme pour les théorèmes \ref{thm 2} et \ref{thm 3}, notre corollaire est la fusion de sa partie torique (voir théorème \ref{thm 4} ci-dessous) et de sa partie elliptique (voir théorème \ref{thm 5} ci-dessous) que l'on montrera respectivement dans les sections \ref{section 3} et \ref{section 4}.

\begin{thm} \label{thm 4}
Soient $k$ un corps de nombres et $a\in k^*$. 
Si $p>2$ est un nombre premier ne divisant pas le discriminant de $k$, alors il existe $c>0$ tel que $h(\alpha) \geq c$ pour tout $\alpha\in k(\langle a \rangle_{\sat,p})^*\backslash \langle a \rangle_\sat$. 
\end{thm}

Le discriminant de $\Q$ étant égal à $1$, ce théorème est alors une généralisation du théorème \ref{thm 1}. 
En particulier, la condition technique "$p\nmid a$ et $p^2 \nmid a^{p-1}-1$"  y a été supprimée. 
 
\begin{thm} \label{thm 5}
Soient $E$ une courbe elliptique définie sur un corps de nombres $k$ et $a\in k^*$. 
Soit $p> 2$ un nombre premier ne divisant pas le discriminant de $k$. 
Si $E$ a une réduction supersingulière en une place de $k$ au-dessus de $p$, alors il existe $c>0$ tel que $\hat{h}_E(P)\geq c$ pour tout $P\in E(k(\langle a\rangle_{\sat, p}))\backslash E_\tors $.
\end{thm}

\begin{rqu} \label{rqu 2}
\rm{Il suffit de montrer ces théorèmes pour $a\in k^*\backslash \mu_\infty$, le cas $a\in \mu_\infty$ ayant déjà été traité dans la remarque \ref{rqu 1} (car dans ce cas, $k(\langle a \rangle_{\sat,p}) \subset k(\mu_\infty)$).}
\end{rqu}

Pour montrer ces théorèmes, on emploiera une nouvelle idée basée sur des théorèmes d'équidistribution que l'on écrira dans la section \ref{section 2}. 
Notons  que depuis, cette technique a été de nouveau utilisée par l'auteur pour traiter d'autres situations de la conjecture \ref{conj 1} \cite{Plessis2, Plessis3}. 

 Pour un groupe $\Gamma$, notons $\Gamma^n=\Gamma \times \dots \times \Gamma$ le produit direct de $n$ copies de $\Gamma$. 
La fusion des deux théorèmes précédents permet d'en déduire le corollaire désiré.

\begin{cor} \label{cor 1}
Soient $E$ une courbe elliptique définie sur un corps de nombres $k$ et $a\in k^*$. 
Posons $\G =\G_m^n \times E$.
Soit $p> 2$ un nombre premier ne divisant pas le discriminant de $k$. 
Si $E$ a une réduction supersingulière en au moins une place de $k$ au-dessus de $p$, alors il existe $c>0$ tel que $\hat{h}_{\G}(\boldsymbol{\alpha}, P)\geq c$ pour tout $(\boldsymbol{\alpha}, P) \in \G(k(\Gamma))\backslash \Gamma_\sat$, où $\Gamma= (\langle a \rangle_{\sat, p})^n \times \{O_E\}$.
\end{cor}

\begin{proof}
 Soit $(\boldsymbol{\alpha},P)=((\alpha_1,\dots,\alpha_n),P)\in \G(k(\Gamma))=\G(k(\langle a \rangle_{\sat, p}))$.
Si la hauteur $\hat{h}_\G(\boldsymbol{\alpha}, P)$ est petite, alors pour tout $i\in \{1,\dots,n\}$, les hauteurs $h(\alpha_i)$ et $\hat{h}_E(P)$ sont également petites.
Des théorèmes \ref{thm 4} et \ref{thm 5}, on a $\alpha_i\in \langle a \rangle_\sat$ pour tout $i$ et $P\in E_\tors$.
Comme $\End(\G_m)=\Z$, il existe, par définition d'un groupe saturé, des entiers $m_1,\dots,m_{n+1} \geq 1$ tels que $[m_i]\alpha_i=\alpha_i^{m_i} \in \langle a \rangle$ pour tout $i$ et $[m_{n+1}]P = O_E$.
En posant $m=\prod_{i=1}^{n+1} m_i$, on obtient $[m](\boldsymbol{\alpha},P)\in \langle a \rangle^n \times \{O_E\} \subset \Gamma$. 
D'où $(\boldsymbol{\alpha},P) \in \Gamma_\sat$ et le corollaire s'ensuit. 
\end{proof}

 Comme $\Gamma \subset \G_\tors$ si et seulement si $\Gamma_\sat= \G_\tors$, il devient clair que $\Gamma_\sat \neq \G_\tors$ dès lors que $a\notin \mu_\infty$, ce qui était l'objectif recherché.

\subsection*{Remerciements}
Je remercie F. Amoroso pour sa patience et sa disponibilité. 
Je tiens aussi à remercier P. Habegger, S. Le Fourn, D. Lombardo et J. Poineau pour avoir pris le temps de répondre à mes questions.
Je remercie également G. R\'emond pour nos discussions sur ce sujet.
Enfin, je remercie également les différents arbitres anonymes qui ont su me guider pour grandement améliorer la qualité de ce papier. 
Ce travail a été financé par Morningside Center of Mathematics, CAS.  

\section{\'Equidistribution} \label{section 2} 
Cette section a pour objectif d'énoncer un r\'esultat d'\'equidistribution que l'on utilisera pour prouver notre th\'eor\`eme \ref{thm 4}.

Dans cette section, $K$ désignera un corps de nombres et $w$ une place finie de $K$. 
Notons $\vert . \vert_w$ la valeur absolue $w$-adique normalisée, i.e. $\vert p\vert_w=p^{-1}$, où $(p)=w \cap \Z$ et $K_w$ le complété de $K$ par rapport à $\vert . \vert_w$. 
La valeur absolue sur $K_w$ s'étend de manière unique en une valeur absolue sur $\overline{K_w}$, et donc de manière unique sur $\C_w$, le complété de $\overline{K_w}$.
 Notons encore $\vert . \vert_w$ la valeur absolue ainsi construite sur $\C_w$.
 
Rappelons que $h : \overline{\Q}^* \to \R$ désigne la hauteur de Weil.  
Elle est positive, invariante sous l'action de $\Gal(\overline{\Q}/\Q)$ et s'annule précisément sur $\mu_\infty$. 
Elle satisfait également l'assertion suivante \[ \forall \alpha,\beta\in \overline{\Q}^*, \; h(\alpha\beta) \leq h(\alpha)+h(\beta)\] ainsi que la propriété de Northcott, i.e. que l'ensemble des points de $\overline{\Q}^*$ de degré et de hauteur bornés est fini. 
Pour des propriétés supplémentaires, le lecteur pourra consulter \cite[§1.5, §1.6]{BombieriGubler}. 

Avant d'énoncer le théorème d'équidistribution que l'on utilisera pour prouver le théorème \ref{thm 4}, nous allons au préalable introduire quelques concepts basiques des espaces analytiques de Berkovich. 
Pour plus de détails sur ces espaces, le lecteur pourra lire \cite{BakerRumely2} ou \cite{Berkovich}. 
 
Notons $\mathbb{P}^1_\mathrm{Berk}(\C_w)$ la droite projective de Berkovich sur $\C_w$. 
C'est un espace topologique compact, Haussdorf et connexe dont les points sont des semi-normes.
La première catégorie de points importante ici est celle s'identifiant à $\C_w$.
Plus précisément, pour chaque $x\in \C_w$, on a une semi-norme $\zeta_x \in \mathbb{P}^1_\mathrm{Berk}(\C_w)$ satisfaisant \[ \forall P\in \C_w[X], \; \zeta_x(P)=\vert P(x)\vert_w,\] où $\C_w[X]$ dénote l'ensemble des polynômes à coefficients dans $\C_w$.  
Posons $D$ le disque unité fermé de $\C_w$. 
Le second point de $ \mathbb{P}^1_\mathrm{Berk}(\C_w)$ qui jouera un rôle crucial ici est le point de Gauss, noté $\zeta_{0,1}$, qui vérifie \[ \forall P\in \C_w[X], \; \zeta_{0,1}(P)=\sup_{z\in D} \vert P(z)\vert_w.\]
Pour un polynôme $\phi(X)\in \C_w[X]$, il est possible de définir une \textit{mesure canonique} sur $\mathbb{P}^1_\mathrm{Berk}(\C_w)$ associée à $\phi$. 
Dans le cas particulier $\phi(X)=X^2$, cette mesure est la masse de Dirac du point de Gauss $\zeta_{0,1}$ \cite[Proposition 10.5]{BakerRumely2}.
En appliquant \cite[Theorem 2.3]{BakerRumely} à $\phi(X)=X^2$, on obtient le théorème d'équidistribution ci-dessous: 

\begin{thm} [Baker-Rumely] \label{thm 6}
Gardons les mêmes notations que ci-dessus et fixons un plongement $\overline{K} \to \overline{K_w}$. 
Soit $(\beta_n)_n$ une suite de points de $\overline{K}^* \backslash \mu_\infty$ telle que $h(\beta_n) \to 0$.
Notons $K_n$ une extension galoisienne finie de $K$ contenant $\beta_n$. 
Alors, pour toute fonction continue $f : \mathbb{P}_\mathrm{Berk}^1(\C_w) \to \R$, on a  
\begin{equation} \label{eq 1}
\frac{1}{[K_n : K]}\sum_{\sigma\in \Gal(K_n/K)} f(\zeta_{\sigma \beta_n}) \to  f(\zeta_{0,1}).
\end{equation}
\end{thm}

\begin{rqu} \label{rqu 3}
\rm{Le terme de gauche dans \eqref{eq 1} ne dépend pas du corps de nombres contenant $\beta_n$. 
Cela se déduit du fait qu'un plongement $ K(\beta_n) \to \overline{K_w}$ se prolonge de $[K_n : K(\beta_n)]$ manières différentes en un plongement $K_n \to \overline{K_w}$.}
\end{rqu}

\begin{rqu} \label{rqu 4}
\rm{Soit $P(X)\in\C_w[X]$ un polynôme. 
La fonction $ \mathbb{P}_\mathrm{Berk}^1(\C_w) \to \R, \; \zeta \mapsto \min\{ \zeta(P), 1/\zeta(P)\}$ est continue par définition de la topologie de Berkovich. 
Cette fonction jouera un rôle important dans la preuve du théorème \ref{thm 4}.}
\end{rqu}

\section{Preuve du théorème \ref{thm 4}}  \label{section 3}
Pour un entier $n\geq 1$, notons $\alpha^{1/n}$ une racine $n$-ième de $\alpha\in \overline{\Q}^*$. 
Rappelons que \[\langle \alpha \rangle_\sat=\{\zeta \; \alpha^x, \zeta\in \mu_\infty, x \in \Q\} \; \text{et} \; \langle \alpha \rangle_{\sat,p}= \{\zeta_{p^r} \alpha^{m/p^s}, r,s\geq 0, m\in \Z\}, \] où $p\geq 2$ désigne un nombre premier. 
Jusqu'à la fin de la section \ref{section 4}, fixons un corps de nombres $k$, un nombre premier $p>2$ ne divisant pas le discriminant de $k$, un nombre $a\in k^* \backslash \mu_\infty$ et une courbe elliptique $E$ définie sur $k$ ayant une réduction supersingulière en une place finie $v$ de $k$ au-dessus de $p$. 

Pour un corps $L$ et un entier $n\geq 0$, posons $L^{p^n}=\{x^{p^n}, x\in L\}$.
Comme $a\in k_v^* \backslash\{1\}$, il existe un entier $\lambda \geq 0$ tel que $a\in k_v^{p^\lambda}\backslash k_v^{p^{\lambda+1}}$. 
Fixons $b\in k_v\backslash k_v^p$ un élément tel que $a=b^{p^\lambda}$.
On peut naturellement identifier $b$ à un élément de $\langle a \rangle_{\sat,p}$ que l'on note encore $b$. 
Posons $K=k(b)$.
Il est clair que $\langle a \rangle_{\sat,p}=\langle b \rangle_{\sat,p}$, $\langle a \rangle_\sat = \langle b \rangle_\sat$ et $K(\langle b \rangle_{\sat,p})=k(\langle a \rangle_{\sat,p})$. 
En vertu de cela et de la remarque \ref{rqu 2}, il suffit de montrer les deux théorèmes ci-dessous pour en déduire les théorèmes \ref{thm 4} et \ref{thm 5}.

\begin{thm} \label{thm 7}
Il existe $c>0$ tel que $h(\alpha) \geq c$ pour tout $\alpha\in K(\langle b \rangle_{\sat,p})^*\backslash \langle b \rangle_\sat$. 
\end{thm}
 
\begin{thm} \label{thm 8}
Il existe $c>0$ tel que $\hat{h}_E(P)\geq c$ pour tout $P\in E(K(\langle b\rangle_{\sat, p}))\backslash E_\tors $.
\end{thm}

\begin{rqu} \label{rqu 5}
\rm{Pour prouver ces théorèmes, on utilisera les groupes de ramification qui est un concept local. 
Or, nos résultats sont globaux. 
Il nous faut donc réussir à passer du local au global. 
L'introduction du corps $K$ nous permettra d'effectuer plus facilement ce passage (lequel est explicité dans \eqref{eq 3}).}
\end{rqu}

\subsection{In\'egalit\'e m\'etrique} \label{ss section 3.1}

Cette sous-section a pour objectif de montrer l'inégalité métrique dont on aura besoin afin de prouver notre théorème \ref{thm 7}; c'est l'objet du lemme \ref{lmm 3}. 
Elle nous sera également utile pour montrer notre théorème \ref{thm 8}. 

Notons que les extensions $K(\zeta_{p^r}, b^{1/p^s})/K$ et $k_v(\zeta_{p^r}, b^{1/p^s})/k_v$ sont galoisiennes et ne dépendent pas du choix de $b^{1/p^s}$ si $r,s$ sont deux entiers tels que $r\geq s\geq 0$.
Notons également que $K(\langle b \rangle_{\sat,p})=K(\zeta_p, b^{1/p}, \zeta_{p^2}, b^{1/p^2}, \dots)$.

\begin{lmm} \label{lmm 1}
Pour tous entiers $r\geq s \geq 0$ avec $r > 0$, l'extension $k_v(\zeta_{p^r}, b^{1/p^s})/k_v$ est totalement ramifiée de degré $(p-1)p^{r+s-1}$.
\end{lmm}

\begin{proof}
L'extension $k_v/\Q_p$ est non ramifiée car $p$ ne divise pas le discriminant de $k$. 
Par ailleurs, $\Q_p(\zeta_{p^r})/\Q_p$ est totalement ramifié de degré $(p-1)p^{r-1}$. 
Par conséquent, $\zeta_p \notin k_v$ et $k_v(\zeta_{p^r})/k_v$ est totalement ramifié de degré $(p-1)p^{r-1}$.
Pour pouvoir conclure, il reste à montrer que $k_v(\zeta_{p^r}, b^{1/p^s})/k_v(\zeta_{p^r})$ est totalement ramifié de degré $p^s$. 

Comme $p>2$, $b\in k_v\backslash k_v^p$ et $\zeta_p\notin k_v$, une conséquence d'un théorème de Schinzel \cite[Proposition 2.5]{viviani} appliquée à $k=k_v$, $a=b$, $m=p$ et $n=p^r$ montre que $b\in k_v(\zeta_{p^r})\backslash k_v(\zeta_{p^r})^p$. 
D'un lemme de Capelli \cite[Chapitre 6, Theorem 9.1]{Lang}, on a \[[k_v(\zeta_{p^r}, b^{1/p^s}) : k_v(\zeta_{p^r})]=p^s.\]
Supposons maintenant que l'extension $k_v(\zeta_{p^r}, b^{1/p^s})/k_v(\zeta_{p^r})$ ne soit pas totalement ramifiée et déduisons-en une absurdité. 
La théorie basique des extensions kummériennes affirme que les sous-corps de cette extension sont les $k_v(\zeta_{p^r}, b^{1/p^g})$ avec $g=0,\dots,s$. 
Le dévissage des extensions de corps locaux \cite[Chapter 1, section 7, Theorem 2]{CasselsFrohlich} permet d'affirmer que 
$k_v(\zeta_{p^r}, b^{1/p})/k_v(\zeta_{p^r})$ est non ramifié. 
La structure des extensions non ramifiées \cite[Chapter 1, section 7]{CasselsFrohlich} donne l'existence d'un entier $m$ tel que $k_v(\zeta_{p^r}, b^{1/p})=k_v(\zeta_m)$.
En particulier, $k_v(\zeta_p, b^{1/p})/k_v$ est abélien. 
Comme $\zeta_p\notin k_v$, un théorème de Schinzel \cite[Theorem 2]{Schinzel} appliqué à $K=k_v$, $a=b$ et $n=p$ montre que $b\in k_v^p$, ce qui est absurde. 
Ceci prouve le lemme.     
\end{proof}

Pour tous entiers $r\geq s\geq 0$, posons pour simplifier $K_{r,s}= K(\zeta_{p^r}, b^{1/p^s})$. 
\begin{lmm} \label{lmm 2}
Il existe une place $w$ de $K$ au-dessus de $v$ telle que 
\begin{enumerate} [i)]
\item $K_w=k_v$; 
\item pour tous entiers $r\geq s \geq 0$ avec $r>0$, il existe une unique place $w_{r,s}$ de $K_{r,s}$ au-dessus de $w$.
\end{enumerate}
\end{lmm}

\begin{proof}
Notons $P(X)\in k[X]$ le polynôme minimal de $b$ sur $k$. 
Comme $b\in k_v$, il existe un polynôme $Q(X)\in k_v[X]$ tel que $P(X)=(X-b)Q(X)$.
Par un lemme de Kummer \cite[Chapter 2, section 10]{CasselsFrohlich} appliqué à $\beta=b$, il existe une place $w$ de $K$ au-dessus de $v$ telle que $K_w=k_v(b)=k_v$. 
Cela montre $i)$. 

Montrons $ii)$. 
Soit $r\geq 1$ un entier. 
L'extension $K_{r,r}/K$ étant galoisienne, le nombre de places de $K_{r,r}$ au-dessus de $w$ est égal à \[ g=\frac{[K_{r,r} : K]}{[K_w(\zeta_{p^r}, b^{1/p^r}) : K_w]}.\]
Clairement, $[K_{r,r} : K] \leq [K(\zeta_{p^r}) : K][K(b^{1/p^r}) : K] \leq (p-1)p^{2r-1}$. 
L'égalité $K_w=k_v$ et le lemme \ref{lmm 1} montrent que $g\leq 1$, i.e. $g=1$.
Le point $ii)$ s'ensuit puisque $K_{r,s}\subset K_{r,r}$ pour tout entier positif $s\leq r$. 
\end{proof}
  
Grâce au $ii)$ du lemme \ref{lmm 2},  et puisque $K(\langle b \rangle_{\sat,p})=\bigcup_{r \geq 0} K_{r,r}$, on peut utiliser sans ambiguïté l'inclusion $K(\langle b \rangle_{\sat, p}) \subset \overline{K_w}$. 
Soit $L\subset K(\langle b \rangle_{\sat,p})$ un corps de nombres contenant $K$.
Afin de simplifier le plus possible notre explication, on notera $w'$ l'unique place de $L$ au-dessus de $w$.

 Comme une conséquence immédiate de ces deux derniers lemmes, on obtient le corollaire  ci-dessous qui nous sera fort utile.

\begin{cor} \label{cor 2}
Soient $r\geq s \geq 0$ deux entiers avec $r>0$. 
Alors $w$ est totalement ramifié dans $K_{r,s}$.
De plus, $[K_{r,s} : K]=(p-1)p^{r+s-1}$. 
\end{cor}

\begin{proof}
Le $i)$ du lemme \ref{lmm 2} donne $K_w=k_v$.
Par le lemme \ref{lmm 1}, l'extension $(K_{r,s})_{w'}/K_w=K_w(\zeta_{p^r}, b^{1/p^s})/K_w$ est donc totalement ramifiée de degré $(p-1)p^{r+s-1}$. 
Le corollaire se d\'eduit maintenant du $ii)$ du lemme \ref{lmm 2}.
\end{proof}

Posons $v(.)=-\log \vert . \vert_v$ sur $k_v^*$. 
Pour des entiers $r\geq s\geq 0$ avec $(r,s)\notin \{(0,0); (1,0)\}$, posons également
\begin{equation*}
G_{r,s}=\begin{cases} 
\Gal(K_{r,s}/K_{r-1,s}) \; \text{si} \; r>s \; \text{et} \; p\mid v(b) \; \text{ou si} \; r>s+1 \; \text{et} \; p\nmid v(b) \\ 
\Gal(K_{r,s}/K_{r,s-1}) \; \text{si} \; r=s \; \text{et} \; p\mid v(b) \; \text{ou si} \; r\leq s+1 \; \text{et} \; p\nmid v(b) \\ 
\end{cases}.
\end{equation*} 
C'est un groupe d'ordre $p$ (donc cyclique) par la multiplicativité des degrés et le corollaire \ref{cor 2}.
Notons $\sigma_{r,s}$ un de ses générateurs et remarquons que le sous-corps $K_{r,s}^{G_{r,s}}$ de $K_{r,s}$ fixé par $G_{r,s}$ est égal à
\begin{equation}\label{eq 2}
K_{r,s}^{G_{r,s}}=\begin{cases} 
K_{r-1,s} \; \text{si} \; r>s \; \text{et} \; p\mid v(b) \; \text{ou si} \; r>s+1 \; \text{et} \; p\nmid v(b) \\ 
K_{r,s-1} \; \text{si} \; r=s \; \text{et} \; p\mid v(b) \; \text{ou si} \; r\leq s+1 \; \text{et} \; p\nmid v(b) \\ 
\end{cases}.
\end{equation}
Nous sommes désormais en mesure d'établir notre inégalité métrique. 

\begin{lmm} \label{lmm 3}
Soient $r\geq s \geq 0$ deux entiers tels que $(r,s) \notin \{(0,0),(1,0)\}$. 
Pour tout $x\in K_{r,s}$ tel que $\vert x \vert_w \leq 1$, on a $\vert \sigma_{r,s} x - x \vert_w \leq p^{-1/p^3}$.
\end{lmm}

\begin{proof}
Rappelons que $K_w=k_v$ (voir le $i)$ du lemme \ref{lmm 2}) est une extension non ramifiée de $\Q_p$ puisque $p$ ne divise pas le discriminant de $k$. 
Par ailleurs, $w$ est totalement ramifié dans $K_{r,s}$ par le corollaire \ref{cor 2}.
L'indice de ramification de l'extension $(K_{r,s})_w/\Q_p$ est donc égal à $[K_{r,s} : K]$.
 
Pour un entier $i\geq 0$, notons  \[ H_i=\left\{\sigma\in \Gal((K_{r,s})_w/K_w), \forall x\in \mathcal{O}_{(K_{r,s})_w}, \; \vert \sigma x - x \vert_w \leq p^{-(i+1)/[K_{r,s} : K]}\right\}\] le $i$-ème groupe de ramification de $\Gal((K_{r,s})_w/K_w)$. 
Il est bien-connu qu'il existe un plus grand entier $T_{r,s}$ tel que $H_{T_{r,s}}$ soit non trivial. 
En appliquant \cite[Corollaire 1.4.4, 1.4.8]{Plessis} avec $F=K_w$, on obtient \[H_{T_{r,s}} = \begin{cases} 
\Gal((K_{r,s})_w/(K_{r-1,s})_w) \; \text{si} \; r>s \; \text{et} \; p\mid v(b) \; \text{ou si} \; r>s+1 \; \text{et} \; p\nmid v(b) \\ 
\Gal((K_{r,s})_w/(K_{r,s-1})_w) \; \text{si} \; r=s \; \text{et} \; p\mid v(b) \; \text{ou si} \; r\leq s+1 \; \text{et} \; p\nmid v(b) \\ 
\end{cases}.\]
L'unicité de la place de $K_{r,s}$ au-dessus de $w$ (voir le $ii)$ du lemme \ref{lmm 2}) entraîne que le morphisme de restriction induit un isomorphisme 
\begin{equation} \label{eq 3}
\Gal((K_{r,s})_w/\left(K_{r,s}^{G_{r,s}}\right)_w) \simeq \Gal(K_{r,s}/K_{r,s}^{G_{r,s}})=G_{r,s}. 
\end{equation}
Par \eqref{eq 2}, on en déduit l'isomorphisme $H_{T_{r,s}} \simeq G_{r,s}$.
Le morphisme $\sigma_{r,s}$ se prolonge donc en un élément de $H_{T_{r,s}}$. 
Soit $x\in K_{r,s}$ tel que $\vert x \vert_w \leq 1$. 
Ainsi, $x\in \mathcal{O}_{(K_{r,s})_w}$ et donc, $\vert \sigma_{r,s} x - x\vert_w \leq p^{- (T_{r,s}+1)/[K_{r,s} : K]}$. 
Le lemme résulte maintenant de l'inégalité  $p^3(T_{r,s}+1) \geq [K_{r,s} : K]$ \cite[Proposition 1.4.9]{Plessis}. 
\end{proof}

\subsection{Preuve du théorème \ref{thm 4}} \label{ss section 3.2}
Avant de pouvoir montrer le théorème \ref{thm 7}, on aura besoin au préalable d'un argument de descente.
Pour un entier $m\geq 1$, notons $\mu_m$ l'ensemble des racines de l'unité d'ordre divisant $m$. 
Rappelons que $K(\langle b \rangle_{\sat,p}) \subset \overline{K_w}$ et que $w'$ d\'esigne l'unique place au-dessus de $w$ de tout corps de nombres inclus dans $K(\langle b \rangle_{\sat,p})$ et contenant $K$.
\begin{prop} \label{prop 1}
Soit $(\alpha_n)_n$ une suite de points de $K(\langle b \rangle_{\sat,p})^*$ telle que $h(\alpha_n) \to 0$. 
Pour tout $n$, notons $r_n \geq s_n \geq 0$ deux entiers avec $r_n\geq 3$ tels que $\alpha_n \in K_{r_n,s_n}$. 
Alors, pour tout $n$ assez grand, il existe $g_n\in\langle b \rangle_\sat$ tel que $\alpha_n \;g_n\in K_{r_n,s_n}^{G_{r_n,s_n}}$.
\end{prop}

\begin{proof}
Pour alléger les notations, posons $G_n = G_{r_n,s_n}$, $\sigma_n=\sigma_{r_n,s_n}$ et $K_n=K_{r_n,s_n}$. 
Si $\alpha_n\in\mu_\infty$ pour un certain $n$, alors $g_n=\alpha_n^{-1} \in \langle b \rangle_\sat$ convient. 
On se réduit ainsi au cas où $\alpha_n \notin \mu_\infty$ pour tout $n$.
Comme $\alpha_n \; g_n\in K_n^{G_n}$ si et seulement si $\alpha_n^{-1} \; g_n^{-1}\in K_n^{G_n}$, alors quitte à remplacer $\alpha_n$ par $\alpha_n^{-1}$, on peut également se ramener au cas où $\vert \alpha_n\vert_w \leq 1$ pour tout $n$. 

Par d\'efinition de la hauteur de Weil, nous avons 
\begin{align*}
h(\alpha_n) = h(\alpha_n^{-1}) & = \frac{1}{[K_n : \Q]} \sum_{v'} [(K_n)_{v'} : \Q_{v'}] \log\max\{1; \vert \alpha_n\vert_{v'}^{-1}\} \\
& \geq \frac{[(K_n)_{w'} : \Q_p]}{[K_n : \Q]} \log\max\{1; \vert \alpha_n\vert_{w'}^{-1}\}
\end{align*}
o\`u dans la somme ci-dessus, $v'$ parcourt l'ensemble des places de $K_n$. 
Le corollaire \ref{cor 2} donne $[(K_n)_{w'} : \Q_p] = [K_n : K][K_w : \Q_p]$. 
Comme $\vert \alpha_n \vert_w \leq 1$, il en r\'esulte que $h(\alpha_n)\geq -\log \vert \alpha_n\vert_w/[K : \Q] \geq 0$. 
Le fait que $h(\alpha_n)\to 0$ entraîne $\vert \alpha_n\vert_w \to 1$.

Supposons par l'absurde qu'il existe une infinité d'entiers $n$ telle que $\beta_n:= (\sigma_n \alpha_n)/\alpha_n \notin \mu_\infty$, i.e. pour tout $n$ quitte à passer à une sous-suite.

Soit $\sigma \in \Gal(K_n/K)$.
Par le $ii)$ du lemme \ref{lmm 2}, on a $\sigma^{-1} w'=w'$. 
D'où 
\begin{equation} \label{eq 4}
\vert \sigma \beta_n -1 \vert_w=\vert \beta_n -1\vert_{\sigma^{-1} w'}=\vert \beta_n-1\vert_w.
\end{equation}
Comme $h(\beta_n) \to 0$ puisque $h(\beta_n) \leq h(\sigma_n \alpha_n)+h(\alpha_n) = 2h(\alpha_n)$ et $\beta_n \notin \mu_\infty$ pour tout $n$, on peut alors appliquer le théorème \ref{thm 6} à cette suite.
 
La fonction $f : \zeta \mapsto \min\{\zeta(X-1), 1/\zeta(X-1)\}$ est continue sur $\mathbb{P}_\mathrm{Berk}^1(\C_w)$ par la remarque \ref{rqu 4}. 
Grâce à \eqref{eq 4}, on obtient, pour tout $\sigma\in \Gal(K_n/K)$, que \[f(\zeta_{\sigma \beta_n})=\min\{\vert \sigma \beta_n-1\vert_w; \vert \sigma  \beta_n-1\vert_w^{-1}\}=\min\{ \vert \beta_n-1\vert_w, \vert \beta_n-1\vert_w^{-1}\}.\]
Par ailleurs, $f(\zeta_{0,1})=1$ (voir section \ref{section 2} pour la définition de $\zeta_{0,1}$). 
Le théorème \ref{thm 6} nous permet enfin d'en d\'eduire que \[ \min\{\vert \beta_n - 1 \vert_w, \vert \beta_n-1\vert_w^{-1}\} = \frac{1}{[K_n : K]} \sum_{\sigma\in \Gal(K_n/K)} f(\zeta_{\sigma \beta_n}) \to f(\zeta_{0,1})=1.\]
La contradiction souhaitée s'obtient en passant à la limite à chaque extrémité de la chaîne d'inégalités ci-dessous (celle de droite découle du lemme \ref{lmm 3} car $\vert \alpha_n\vert_w \leq 1$) : \[ \min\{\vert \beta_n - 1 \vert_w, \vert \beta_n-1\vert_w^{-1}\} \leq \vert\beta_n - 1 \vert_w =\vert \alpha_n\vert_w^{-1} \vert \sigma_n \alpha_n - \alpha_n\vert_w \leq p^{-1/p^3} \vert \alpha_n\vert^{-1}_w. \] 

Nous allons maintenant montrer que $\beta_n \in \mu_p$. 
Posons $m=p^c d$ avec $p\nmid d$ l'ordre de $\beta_n\in K_n$ et $\gamma= \beta_n^{p^c}$. 
Clairement, $\gamma$ est d'ordre $d$.
Comme $p\nmid d$, alors $w$ est non ramifié dans $K(\gamma)$ \cite[Chapter II, Proposition 7.12]{Neukirch}. 
Ensuite, $w$ est totalement ramifié dans $K_n$ d'après le corollaire \ref{cor 2}.
Cela prouve l'égalité $K_n \cap K(w)=K$. 
En particulier, $w\in K$ et donc, $\sigma_n \gamma=\gamma$. 
Comme $\sigma_n\alpha_n^{p^c}=\gamma\alpha_n^{p^c}$, on obtient par récurrence que $\sigma_n^i\alpha_n^{p^c}=\gamma^i\alpha_n^{p^c}$ pour tout $i\geq 1$.
Le morphisme $\sigma_n$ étant d'ordre $p$, il s'ensuit en prenant $i=p$ que $\gamma^p=1$.
Or, $\gamma$ est d'ordre $d$ avec $p\nmid d$. 
D'où $d=1$.
Ainsi, $\beta_n\in K_n$ est d'ordre une puissance de $p$. 

Par la multiplicativité des degrés et le corollaire \ref{cor 2}, on a \[[K_{r_n+1,s_n} : K_n]=[K(\zeta_{p^{r_n+1}}, b^{1/p^{s_n}}) : K(\zeta_{p^{r_n}}, b^{1/p^{s_n}})] =p,\] ce qui montre que $\zeta_{p^{r_n+1}}\notin K_n$.
D'où $\beta_n \in \mu_{p^{r_n}}$. 
Par \eqref{eq 2}, les racines $p^{r_n-1}$-ième de l'unité appartiennent à $K_n^{G_n}$.
 Par conséquent, $\beta_n^p \in K_n^{G_n}$, i.e. $\sigma_n \beta_n^p=\beta_n^p$. 
De plus, $\sigma_n \alpha_n^p=\beta_n^p\alpha_n^p$.
De ces deux égalités, on en déduit par récurrence que $\sigma_n^i \alpha_n^p=\beta_n^{pi}\alpha_n^p$ pour tout $i\geq 1$.
En prenant $i=p$, il en résulte que $\beta_n\in \mu_{p^2}$.
Par hypothèse, $r_n\geq 3$.
Ainsi, $\mu_{p^2}\subset \mu_{p^{r_n-1}}\subset K_n^{G_n}$.
Par conséquent, $\sigma_n\beta_n = \beta_n$.
Comme $\sigma_n\alpha_n=\beta_n \alpha_n$, il s'ensuit, toujours par récurrence,  que $\sigma_n^i \alpha_n = \beta_n^i \alpha_n$ pour tout $i\geq 1$.
En prenant à nouveau $i=p$, on en conclut que $\beta_n\in\mu_p$.

Nous pouvons maintenant conclure notre preuve. 
Posons \[g_n=\begin{cases} 
\zeta_{p^{r_n}} \; \text{si} \; r_n>s_n \; \text{et} \; p\mid v(b) \; \text{ou si} \; r_n>s_n+1 \; \text{et} \; p\nmid v(b) \\ 
b^{1/p^{s_n}} \; \text{si} \; r_n=s_n \; \text{et} \; p\mid v(b) \; \text{ou si} \; r_n\leq s_n+1 \; \text{et} \; p\nmid v(b)
\end{cases}.\] 
D'après \eqref{eq 2}, on a $g_n\notin K_n^{G_n}$ et $g_n^p \in K_n^{G_n}$.
D'où, $(\sigma_n g_n)/g_n \in \mu_p \backslash \{1\}$.
Le groupe $\mu_p$ étant cyclique d'ordre $p$, il existe alors un entier $l$ tel que $ \beta_n = ((\sigma_n g_n)/g_n)^l$. 
Un simple calcul montre que $\sigma_n(\alpha_n \; g_n^{-l})=\alpha_n \; g_n^{-l}$, i.e. $\alpha_n \; g_n^{-l}\in K_n^{G_n}$.
La proposition s'ensuit puisque $g_n\in \langle b \rangle_\sat$.  
\end{proof}

\textit{Démonstration du théorème \ref{thm 7}} :
Soit $(\alpha_n)_n$ une suite de points de $K(\langle b \rangle_{\sat,p})^*$ telle que $h(\alpha_n) \to 0$. 
 Posons $\N=\{0,1,\dots\}$ et pour tout $n$, posons aussi \[ \Lambda_n = \{(r,s) \in \N^2, r\geq s \; \vert \; \exists g \in \langle b \rangle_\sat, \; \alpha_n \; g \in K_{r,s}\}.\]
C'est un ensemble non vide puisque $K(\langle b \rangle_{\sat,p})=\bigcup_{r\geq 0} K_{r,r}$. 
Il existe donc un élément $(r_n,s_n)\in \Lambda_n$ qui soit minimal pour l'ordre lexicographique usuel de $\N^2$. 
Soit $\zeta(n) \; b^{x_n} \in \langle b \rangle_\sat$ tel que $\alpha_n \; \zeta(n) \; b^{x_n} \in K_{r_n,s_n}$, où $\zeta(n)\in \mu_\infty$ et $x_n\in \Q$. 

Supposons par l'absurde que la suite $(r_n)_n$ ne soit pas bornée. 
Quitte à en extraire une sous-suite, on peut supposer que $r_n \to +\infty$ et $r_n\geq 3$ pour tout $n$. 
Le corps $\Q$ étant archimédien, il existe un nombre $y_n \in \Q$ vérifiant $\vert y_n \vert < p^{-(r_n-2)}$ ainsi qu'un entier $q_n\in \Z$ tels que $x_n=y_n+q_n p^{-(r_n-2)}$. 
En remarquant que $b^{1/p^{r_n-2}} \in K_{r_n, r_n-2}$, on en déduit que $\alpha_n \; \zeta(n)\; b^{y_n} \in K_{r_n, s'_n}$, où $s'_n=\max\{s_n,r_n-2\} \leq r_n$. 
Les propriétés basiques de la hauteur de Weil donnent 
\begin{equation} \label{eq 5}
0 \leq  h(\alpha_n \; \zeta(n) \; b^{y_n}) \leq h(\alpha_n)+ p^{-(r_n-2)}h(b).
 \end{equation}
Le terme de droite tend vers $0$ quand $n$ tend vers $+\infty$. 
Par le théorème des gendarmes, $h(\alpha_n \zeta(n) b^{y_n}) \to 0$. 
En appliquant maintenant la proposition \ref{prop 1} avec la suite $(\alpha_n\zeta(n) b^{y_n})_n$, on en déduit l'existence d'un $g_n\in \langle b \rangle_\sat$ tel que $\alpha_n \zeta(n) b^{y_n} g_n \in K_{r_n, s'_n}^{G_{r_n, s'_n}}$ pour tout $n$ assez grand. 
Comme $\zeta(n)b^{y_n} g_n \in \langle b \rangle_\sat$ et que par \eqref{eq 2}, soit $K_{r_n,s'_n}^{G_{r_n,s'_n}}=K_{r_n-1,s'_n}$ et $r_n>s'_n$, soit $K_{r_n,s'_n}^{G_{r_n,s'_n}}=K_{r_n,s'_n-1}$ et $r_n \leq s'_n+1$, il en résulte que soit $(r_n-1,s'_n)\in \Lambda_n$, soit $(r_n, s'_n-1)\in \Lambda_n$. 
La première possibilité contredit la minimalité de $(r_n,s_n)$. 
On a donc $K_{r_n,s'_n}^{G_{r_n,s'_n}}=K_{r_n,s'_n-1}$ et $r_n \leq s'_n+1$. 
Mais alors, $s'_n=s_n$ par définition de ce dernier. 
D'où $(r_n, s_n-1)\in \Lambda_n$ contredisant à nouveau la minimalité de $(r_n,s_n)$. 

La suite $(r_n)_n$ est donc majorée, disons par un entier $r$. 
Rappelons que $r_n\geq 3$ pour tout $n$.
Par construction de la suite $(\alpha_n \; \zeta(n) \; b^{y_n})_n$ et par \eqref{eq 5}, on a \[ \alpha_n \; \zeta(n) \; b^{y_n} \in K_{r,r} \; \text{et} \; h(\alpha_n \; \zeta(n) \; b^{y_n}) \leq h(\alpha_n)+p^{-1} h(b)\] pour tout entier $n$. 
Le théorème de Northcott montre que la suite $(\alpha_n \; \zeta(n) \; b^{y_n})_n$ ne contient qu'un nombre fini de termes distincts.
Notons les $\{\beta_1, \dots, \beta_m\}$. 
 Ainsi, pour tout $n$, il existe $i_n\in\{1,\dots,m\}$ tel que $\alpha_n \; \zeta(n) \; b^{y_n} = \beta_{i_n}$.
Quitte à retirer les premiers termes de la suite $(\alpha_n\zeta(n)b^{y_n})_n$, on peut supposer que pour tout $i\in\{1,\dots,m\}$, le terme $\beta_i$ apparaît une infinité de fois dans cette suite. 

Soient $i\in \{1,\dots,m\}$ et $(\alpha_{\phi(n)} \;\zeta(\phi(n)) \; b^{y_{\phi(n)}})_n$ une sous-suite (infinie) constante égale à $\beta_i$.
Pour tout $n$, il existe $g'_n \in \langle b \rangle_\sat$ tel que $\alpha_{\phi(n)}=\alpha_{\phi(1)} g'_n$.
Clairement, 
 \[0 \leq \inf\{h(\gamma\alpha_{\phi(1)}), \gamma \in \langle b \rangle_\sat \}= \inf\{h(\gamma\alpha_{\phi(n)}), \gamma \in \langle b \rangle_\sat \} \leq h(\alpha_{\phi(n)}).\] 
Par le théorème des gendarmes, on en conclut que \[\inf\{h(\gamma\alpha_{\phi(1)}), \gamma \in \langle b \rangle_\sat \}=0,\] ce qui ne peut se produire que si $\alpha_{\phi(1)}\in \langle b \rangle_\sat$ d'après \cite[Lemma 2.4]{Pottmeyer}. 
On a donc $\beta_i= \alpha_{\phi(1)} \;\zeta(\phi(1)) \; b^{y_{\phi(1)}} \in \langle b\rangle_\sat$. 
Ceci étant valable pour tout $i\in\{1,\dots,m\}$, il s'ensuit que $\alpha_n = \beta_{i_n} \zeta(n)^{-1} b ^{-y_n}\in \langle b \rangle_\sat$ pour tout $n$. 
Le théorème s'ensuit.

\section{Preuve du théorème \ref{thm 5}} \label{section 4}
Rappelons que $E$ est une courbe elliptique définie sur le corps de nombres $K$ (voir le d\'ebut de la section \ref{section 3}) ayant une réduction supersingulière en $w$. 
On peut donc la représenter par une équation de Weierstrass de la forme $y^2=x^3+Ax+B$, où $A, B \in K$ tels que $\vert A\vert_w, \vert B\vert_w \leq 1$ et $\vert 2^4(4A^3+27B^2)\vert_w=1$.
Les coordonnées d'un point $P=(x(P), y(P))\in E(\C_w)$ seront toujours par rapport à ce modèle. 
 
Pour montrer notre théorème \ref{thm 8}, on procèdera comme pour le théorème \ref{thm 7}, i.e. que l'on prouvera d'abord une inégalité métrique (qui est l'objet du prochain résultat), laquelle nous servira ensuite à établir une descente. 

Le lemme ci-dessous ne recquiert que la bonne r\'eduction de $E$ en $w$.
\begin{lmm} \label{lmm 4}
Soient $r\geq s\geq 0$ deux entiers avec $(r,s)\notin\{(0,0); (1,0)\}$ et $P\in E(K_{r,s})$ tels que $\sigma_{r,s} P - P\neq 0$. 
Alors $\vert x''\vert_w \geq p^{2/p^3}$, où $\sigma_{r,s} P-P=(x'',y'')$.
\end{lmm}

\begin{proof}
Posons $\vert.\vert=\vert . \vert_w$, $P=(x,y)$ et $\sigma_{r,s} P=(x',y')$.
Notons $\widetilde{E}$ la réduite de $E$ modulo $w$ et $\widetilde{Q}$ la réduction modulo $w'$ d'un point $Q\in E(K_{r,s})$. 

\underline{Premier cas} : $\vert x\vert >1$.  
Pour tout $(X,Y)\in E(K_{r,s})$ tel que $\vert X\vert > 1$, on a
\begin{equation} \label{eq 6}
 \vert Y^2\vert=\vert X^3 +AX+B\vert= \vert X^3\vert
 \end{equation}
car $\vert A\vert, \vert B\vert \leq 1$.
Posons $t=-x/y$ et $t'=-x'/y'=\sigma_{r,s} t$. 
La relation \eqref{eq 6} appliquée à $(X,Y)=(x,y)$ entraîne $\vert t\vert= \vert x\vert^{-1/2}<1$. 
D'où $\vert t'-t\vert \leq p^{-1/p^3}$ par le lemme \ref{lmm 3}. 

Notons $\mathcal{O}_{K_w}\ldbrack X_1,\dots,X_n\rdbrack$ l'anneau des séries formelles à $n$ indéterminées à coefficients dans $\mathcal{O}_{K_w}$, $F(X_1,X_2)\in\mathcal{O}_{K_w} \ldbrack X_1,X_2 \rdbrack$ la loi formelle de $E/K_w$ associée au paramètre local $-X/Y$ à l'origine et $i(X_1)\in \mathcal{O}_{K_w}\ldbrack X_1 \rdbrack$ l'unique série entière telle que $F(X_1,i(X_1)) \equiv 0$. 

Comme $\sigma_{r,s} P-P=(x'',y'')$, alors $t'':=-x''/y''=F(t',i(t))$. 
Par ailleurs, $\widetilde{P}=\widetilde{O_E}$ puisque $\vert x \vert > 1$. 
Par conséquent, $\widetilde{\sigma_{r,s} P-P}=\widetilde{O_E}$, i.e. $\vert x''\vert > 1$. 
La relation \eqref{eq 6} appliquée à $(X,Y)=(x'',y'')$ prouve que $\vert t''\vert = \vert x''\vert ^{-1/2}$ et l'identité $F(X_1, i(X_1))\equiv 0$ fournit l'existence d'une série entière $G(X_1,X_2)\in \mathcal{O}_{K_w} \ldbrack X_1,X_2\rdbrack$ telle que $F(X_1,X_2)=(i(X_1)-X_2)G(X_1,X_2)$. 
En spécialisant avec $X_1=t'$ et $X_2=i(t)$, on obtient, car $\vert t \vert, \vert t'\vert \leq 1$ et $i(X_1) \in \mathcal{O}_{K_w}\ldbrack X_1\rdbrack$ : 
\begin{equation*}
\vert t''\vert = \vert i(t')-i(t)\vert \; \vert G(t', i(t))\vert \leq \vert i(t')-i(t)\vert \leq \vert t'-t\vert \leq p^{-1/p^3}.
\end{equation*}
Comme $\vert t''\vert=\vert x''\vert^{-1/2}$, il s'ensuit que $\vert x''\vert \geq p^{2/p^3}$. 
Cela achève ce cas.

\underline{Second cas} : $\vert x\vert \leq 1$. 
Comme $y^2=x^3+Ax+B$, alors $\vert y\vert \leq 1$ et donc, d'après le lemme \ref{lmm 3}, $\vert x'-x\vert \leq p^{-1/p^3}$ et $\vert y'-y\vert \leq p^{-1/p^3}$. 
D'où \[\vert x' - x\vert^{-2} \geq p^{2/p^3} \; \text{et} \; \vert y' - y\vert^{-2} \geq p^{2/p^3}.\]
Comme $E$ a bonne r\'eduction en $w$, le point $\widetilde{P}$ est donc non singulier sur $\widetilde{E}$.
Ainsi, soit $\vert y\vert=1$, soit $\vert 3x^2+A\vert=1$. 

\underline{Premier sous-cas} : $\sigma_{r,s} P = - P$. 
Cela se traduit par les relations $x'=x$ et $y' = -y$. 
Si $y=0$, on aurait alors $\sigma_{r,s} P=P$, ce qui est absurde par hypoth\`ese. 
D'o\`u $y\neq 0$. 
Par ailleurs, on a \'egalement $\vert y'-y\vert = \vert 2y \vert = \vert y \vert \leq p^{-1/p^3}$ puisque $w$ est au-dessus de $p\geq 5$.
 Par ce qui pr\'ec\`ede, on a $\vert 3x^2+A\vert=1$. 

Comme $-2y\neq 0$ et $- P = (x,-y)$, la formule d'addition de deux points sur une courbe elliptique \cite[Chapter III, Algorithm 2.3]{Silverman2} donne la relation suivante : \[ x''= \left( \frac{3x^2+A}{-2y}\right)^2 -2x. \] 
L'in\'egalit\'e ultram\'etrique fournit $\vert x''\vert = \vert y\vert^{-2} \geq p^{2/p^3}$, ce qui ach\`eve ce sous-cas.

\underline{Second sous-cas} : $\sigma_{r,s} P \neq - P$. 
Comme $\sigma_{r,s} P \neq P$, on a alors $x'\neq x$.
Comme $x'\neq x$ et $- P = (x,-y)$, la formule d'addition de deux points sur une courbe elliptique nous permet d'en d\'eduire que 
\begin{equation} \label{eq 7}
x''=\left(\frac{y'+y}{x'-x}\right)^2-x'-x.
\end{equation}
\underline{$i)$} : $\vert y+y' \vert \geq 1$. 
Par le $ii)$ du lemme \ref{lmm 2}, on a l'égalité $\sigma_{r,s}^{-1} w' = w'$. 
D'où \[\vert y'\vert= \vert \sigma_{r,s} y\vert_w= \vert \sigma_{r,s} y\vert_{w'}=\vert y \vert_{\sigma^{-1}_{r,s} w'} = \vert y\vert \leq 1.\] 
De même, $\vert x'\vert = \vert x \vert \leq 1$. 
Ainsi, $\vert y+y'\vert=1$ et $\vert x+x'\vert \leq 1$.
 En passant à la valeur absolue dans \eqref{eq 7}, on obtient l'inégalité désirée car \[ p^{2/p^3} \leq \vert x'-x\vert^{-2} = \max\left\{\left\vert \frac{y+y'}{x'-x}\right\vert^2, \vert x'+x\vert\right\}= \vert x''\vert.\]
\underline{$ii)$} : $\vert y'+y\vert < 1$.
Rappelons que $\vert x'-x\vert, \vert y'-y\vert \leq p^{-1/p^3}<1$ et $p> 2$.
Ainsi, \[ \vert y\vert=\vert 2y\vert=\vert (y'+y)-(y'-y)\vert \leq \max\{\vert y'+y\vert, \vert y'-y\vert\} <1.\] 
On obtient donc $\vert 3x^2+A\vert =1$. 
Or, $\vert x'-x\vert < 1$, i.e. $x'\equiv x \mod w$. 
D'où $\vert x'^2+x'x+x^2+A\vert= \vert 3x^2 + A\vert= 1$. 
Ensuite, en soustrayant l'égalité $y^2=x^3+Ax+B$ à l'égalité $y'^2=x'^3+Ax'+B$ on obtient, après factorisation et division par $x'-x\neq 0$,\[\frac{(y'-y)(y'+y)}{x'-x}=x'^2+x'x+x^2+A.\]
D'où $\left\vert (y'+y)/(x'-x) \right\vert = 1/\vert y'-y\vert$. 
Comme $\vert x'+x\vert \leq 1$, on déduit de \eqref{eq 7} que \[ p^{2/p^3} \leq \vert y-y'\vert ^{-2} = \max\left\{\left\vert \frac{y+y'}{x'-x}\right\vert^2, \vert x'+x\vert\right\} = \vert x''\vert .\]
Cela montre la proposition.  
\end{proof}

Une fois l'inégalité métrique obtenue, il nous faut maintenant effectuer un argument de descente.
Comme nous le verrons dans la preuve du théorème \ref{thm 8}, cet argument sera plus simple (bien que similaire) à mettre en place que celui du théorème \ref{thm 7}. 
Cela est dû au fait que $E(K(\langle b \rangle_{\sat,p}))$ ne contienne aucun point de torsion d'ordre une puissance de $p$; c'est l'objet du lemme \ref{lmm 6}.

Mais avant, énonçons un lemme qui découle immédiatement du $ii)$ du lemme \ref{lmm 2} et de \cite[Corollary 1.3.2]{BombieriGubler}. 

\begin{lmm} \label{lmm 5}
Pour tout $Q\in E(K(\langle b\rangle_{\sat,p}))$, on a $[K(Q) : K]=[K(Q)_w : K_w]$. 
\end{lmm}

Notons $E[N]$ le groupe des points de torsion de $E(\overline{K})$ tu\'e par un entier $N\geq 1$. 
  
\begin{lmm} \label{lmm 6}
On a $E(K_{r,s}) \cap E[p]=\{O_E\}$ pour tous entiers $r\geq s\geq 0$ avec $r>0$.
\end{lmm}
 
\begin{proof}
Supposons par l'absurde qu'il existe un point $Q\in E[p] \cap E(K_{r,s})$ d'ordre $p$.
La multiplicativité des degrés et le corollaire \ref{cor 2} montrent que $p^2-1$ ne divise pas le degré $[K(Q) : K]$. 

Par le lemme \ref{lmm 2}, $K_w=k_v$.
Cela fournit l'égalité $K(Q)_w=k_v(Q)$.
 Nous souhaitons obtenir une contradiction en montrant que $p^2-1$ divise $[k_v(Q) : k_v]= [K(Q)_w : K_w]=[K(Q) : K]$, la dernière égalité provenant du lemme \ref{lmm 5}. 

Pour un entier $n\geq 1$, notons $\F_{p^n}$ le corps fini à $p^n$ éléments et $k'_v$ l'unique extension non ramifiée de degré $2$ sur $k_v$.
Fixons dès maintenant un isomorphisme $E[p]\simeq (\Z/p\Z)^2$ de telle sorte que $Q$ s'envoie sur $\begin{pmatrix}
1 \\
0 \\
\end{pmatrix}$. 
Posons $C=\Gal(k'_v(E[p])/k'_v)$ et $\rho : C \hookrightarrow \mathrm{GL}_2(\F_p)$ la représentation naturelle, où $\mathrm{GL}_2(\F_p)$ désigne le groupe des matrices carrées inversibles de taille $2$ à coefficients dans $\F_p$.

 Par construction, le corps résiduel de $k'_v$ contient $\F_{p^2}$ (l'inertie de $k'_v/k_v$ est égale à $2$) et est une extension non ramifiée de $\Q_p$ ($k'_v/k_v$ est non ramifié et $k_v/\Q_p$ l'est également car $p$ ne divise pas le discriminant de $k$). 
Rappelons que $E$ a une réduction supersingulière sur $k'_v$. 
Le groupe $\rho(C)$ est donc un sous-groupe de Cartan non d\'eploy\'e, i.e. un groupe cyclique d'ordre $p^2-1$ \cite[Proposition 12, d)]{Serre}.
L'injectivité de $\rho$ fournit l'égalité 
\begin{equation} \label{eq 8}
[k'_v(E[p]) : k'_v]=p^2-1.
\end{equation}
Notons $\mathrm{Stab} \; Q = \left\{M\in \mathrm{GL}_2(\F_p), M\begin{pmatrix}
1 \\
0 \\
\end{pmatrix} = \begin{pmatrix}
1 \\
0 \\
\end{pmatrix}\right\} $ le stabilisateur de $Q$ pour l'action naturelle de $\mathrm{GL}_2(\F_p)$ sur $E[p]=(\Z/p\Z)^2$.
Il est clair que $\mathrm{Stab} \; Q = \begin{pmatrix}
1 & * \\
0 & * \\
\end{pmatrix}$. 
Soit $M\in \rho(C) \cap \mathrm{Stab} \; Q$.
Alors $M=\begin{pmatrix}
1 & t \\
0 & u \\
\end{pmatrix}$ pour certains $(t,u)\in \F_p \times \F_p$.

Supposons par l'absurde que $u\neq 1$. 
Remarquons que $(1+u)^2-4u=(u-1)^2$ et rappelons que $p>2$.
Par \cite[§ 2.1 b)]{Serre}, $M$ appartient \`a un sous-groupe de Cartan et un seul. 
De plus, celui-ci se doit d'\^etre d\'eploy\'e. 
L'absurdit\'e provient du fait que $M\in \rho(C)$ qui est non-d\'eploy\'e. 
On a donc $u=1$.

Par conséquent, $M^i= \begin{pmatrix}
1 & it \\
0 & 1 \\
\end{pmatrix}$ pour tout entier $i\geq 1$. 
L'ordre de $M$ divise ainsi $p$. 
Par ailleurs, il divise également $p^2-1$ puisque $M\in \rho(C)$; c'est donc la matrice identité.
Le groupe $\rho(C) \cap \mathrm{Stab} \; Q$ étant trivial, l'inclusion \[\rho(\rho^{-1}(\mathrm{Stab} \; Q))= \rho(C\cap \rho^{-1} (\mathrm{Stab} \; Q)) \subset \rho(C) \cap \mathrm{Stab} \; Q\] et l'injectivité de $\rho$ prouvent que le groupe $\rho^{-1}( \mathrm{Stab}\;  Q)$ est trivial. 
Comme $k'_v(Q)$ est égal au sous-corps de $k'_v(E[p])$ fixé par $\rho^{-1}( \mathrm{Stab} \; Q)$, il s'ensuit que $k'_v(Q)=k'_v(E[p])$ et \eqref{eq 8} donne l'égalité $[k'_v(Q) : k'_v]=p^2-1$. 

On a ainsi établi que $[k_v(Q) : k_v] \geq [k'_v(Q) : k'_v]= p^2-1$. 
Par ailleurs, toujours d'après  \cite[Proposition 12, d)]{Serre} (appliquée cette fois-ci avec $K=k_v$), il en résulte que $[k_v(E[p]) : k_v] \in\{p^2-1, 2(p^2-1)\}$. 
De tout cela et de la multiplicativité des degrés, on obtient $[k_v(Q) : k_v]\in \{p^2-1, 2(p^2-1)\}$. 
Ceci prouve le lemme. 
\end{proof} 
 
Rappelons que $\hat{h}_E : E(\overline{K}) \to \R$ désigne la hauteur de Néron-Tate. 
Elle est positive, invariante sous l'action de $\Gal(\overline{K}/K)$, s'annule sur $E_\tors$ et satisfait la propriété de Northcott, i.e. que l'ensemble des points de $E(\overline{K})$ de degré et de hauteur bornés est fini. 
Elle satisfait également la loi du parallélogramme, i.e. : \[ \forall P,Q\in E(\overline{K}),\; \hat{h}_E(P+Q)+\hat{h}_E(P-Q)=2(\hat{h}_E(P)+\hat{h}_E(Q)).\]
Pour d'autres propriétés en plus de celles-ci, on pourra voir \cite[Chapter VIII, §9]{Silverman2}. 

Pour une place $\nu$ (non n\'ecessairement finie) d'un corps de nombres $F\subset \overline{K}$, on note $\lambda_\nu : E(F_\nu) \to \R$ la hauteur locale de N\'eron sur $E$ associ\'ee \`a $\nu$. 
Toutes les propriétés que l'on utilisera ci-dessous concernant cette hauteur se trouvent dans \cite[Chapter VI]{Silverman}.
Si $E$ a bonne r\'eduction en $\nu$, alors $\lambda_\nu(A)\geq 0$ for all $A\in E(F_\nu)$.

Soient $A\in E(\overline{K})$ et $F/K(A)$ une extension finie. 
Pour une place finie $\nu$ de $K$, posons 
\begin{equation} \label{eq 9}
\hat{h}_\nu(A) = \frac{1}{[F : \Q]}\sum_{\nu' \mid \nu} [F_{\nu'} : \Q_\nu] \lambda_{\nu'} (A),
\end{equation}
o\`u $\nu'$ parcourt toutes les places de $F$ au-dessus de $\nu$.
Par ce qui pr\'ec\`ede, $\hat{h}_\nu(A)\geq 0$ si $E$ a bonne r\'eduction en toute place de $F$ au-dessus de $\nu$. 
Ensuite, posons \[\hat{h}_\infty(A)=\frac{1}{[F : \Q]} \sum_{\nu' \mid \infty} [F_{\nu'} : \R] \lambda_{\nu'}(A),\] où $\nu'$ parcourt toutes les places infinies de $F$.
Il est bien-connu que les termes de droite ci-dessus ne d\'ependent pas de $F$.
Cela découle du fait que $\lambda_{\nu'}(A)=\lambda_{\nu''}(A)$, o\`u $\nu''$ désigne la place de $K(A)$ au-dessous de $\nu'$. 
Par ailleurs, on déduit de \cite[Chapter VI, Theorem 2.1]{Silverman} que $\hat{h}_E(A)= \hat{h}_\infty(A)+\sum_\nu \hat{h}_\nu(A)$, où $\nu$ parcourt l'ensemble des places finies de $K$.  

\begin{lmm} \label{lmm 7}
Soit $(Q_n)_n$ une suite de $E(\overline{K})\backslash E_\tors$ telle que $\hat{h}_E(Q_n) \to 0$. 
Alors $\liminf_{n\to +\infty} \hat{h}_\infty(Q_n) \geq 0$. 
De plus, si $\nu$ est une place finie de $K$, alors on a également $\liminf_{n\to +\infty} \hat{h}_\nu(Q_n) \geq 0$.
Enfin, si $E$ a potentiellement bonne réduction en $\nu$, alors $\hat{h}_\nu(Q_n) \geq 0$ for all $n$. 
\end{lmm}

\begin{proof}
Pour la premi\`ere affirmation, voir \cite[Lemma 8.8]{Habegger}.  

Il existe une extension galoisienne finie $F/K$ telle que $E$ a soit bonne r\'eduction, soit r\'eduction multiplicative d\'eploy\'ee en toute place finie de $F$. 
 
Si $E$ a potentiellement bonne r\'eduction en $\nu$, alors $E$ a bonne réduction en toute place de $F$ au-dessus de $\nu$ et le lemme r\'esulte maintenant du texte suivant \eqref{eq 9}.
Sinon, $E$ a r\'eduction multiplicative d\'eploy\'ee  en toute place de $F$ au-dessus de $\nu$. 
Dans cette situation, nous pouvons recopier la preuve faite dans \cite[Lemma 8.8]{Habegger}, ce qui prouve le lemme.  
\end{proof}

\textit{Démonstration du théorème \ref{thm 8}} :
Soit $(P_n)_n$ une suite de points de $E(K(\langle b \rangle_{\sat,p}))$ telle que $\hat{h}_E(P_n) \to 0$. 
Pour tout $n$, notons $(r_n,s_n)$ le plus petit couple d'entiers positifs (pour l'ordre lexicographique usuel) tel que $r_n \geq s_n$ et $P_n\in E(K_{r_n,s_n})$. 
Un tel couple existe puisque $K(\langle b \rangle_{\sat,p})= \bigcup_{r\geq 1} K_{r,r}$. 

Supposons par l'absurde que $(r_n,s_n)\notin \{(0,0),(1,0)\}$ pour une infinité de $n$, i.e. pour tout $n$ quitte à passer à une sous-suite.
Pour simplifier, posons $\sigma_n=\sigma_{r_n,s_n}$. 

Supposons également par l'absurde qu'il existe une infinité d'entiers $n$ telle que $Q_n=\sigma_n P_n -P_n=(x''_n,y''_n)\notin E_\tors$, i.e. pour tout $n$ quitte à en extraire une sous-suite.
Le fait que $\hat{h}_E$ soit positive, invariante par conjugaison et satisfasse la loi du parallélogramme permet d'obtenir \[0 \leq \hat{h}_E(Q_n)=2(\hat{h}_E(\sigma_n P_n)+\hat{h}_E(P_n))-\hat{h}_E(\sigma_n P_n+P_n)\leq 4\hat{h}_E(P_n).\] 
D'après le théorème des gendarmes, $\hat{h}_E(Q_n) \to 0$.  
Rappelons que $E$ a bonne réduction en $w$ et que $w'$ est l'unique place de $K_{r_n,s_n}$ au-dessus de $w$. 
Ainsi, $\lambda_{w'}(Q_n)=(1/2)\log\max\{1, \vert x(Q_n) \vert_w\}$ et comme $\vert x(Q_n)\vert_w \geq p^{2/p^3}$ d'apr\`es le lemme \ref{lmm 4}, on en d\'eduit que \[\hat{h}_w(Q_n) \geq \frac{[(K_{r_n,s_n})_{w'} : \Q_p] \log p}{p^3[K_{r_n,s_n} : \Q]}.\]
Ensuite, le corollaire \ref{cor 2} affirme que $[(K_{r_n,s_n})_{w'} : \Q_p]=[K_{r_n,s_n} : K][K_w : \Q_p]$. 
D'o\`u  $\hat{h}_w(Q_n) \geq \log p/(p^3[K:\Q])$. 
En conclusion, \[ \hat{h}_E(Q_n) \geq \frac{\log p}{p^3[K: \Q]}+ \hat{h}_\infty(Q_n)+ \sum_{\nu\neq w} \hat{h}_\nu(Q_n),\] o\`u la somme parcourt l'emsemble des places finies de $K$ distinctes de $w$.
Notons $\mathcal{M}$ l'ensemble des places finies de $K$ pour lequel $E$ n'a pas potentiellement bonne réduction. 
En particulier, $w\notin \mathcal{M}$. 
Il est bien connu que $\mathcal{M}$ est fini. 
Le lemme \ref{lmm 7} permet d'en déduire que \[ \hat{h}_E(Q_n) \geq \frac{\log p}{p^3[K: \Q]}+ \hat{h}_\infty(Q_n)+ \sum_{\nu\in\mathcal{M}} \hat{h}_\nu(Q_n).\] 
L'absurdit\'e s'ensuit car le terme de gauche converge vers $0$ tandis que la limite inf\'erieure dans celui de droite est d'au moins $\log p/(p^3[K:\Q])$ gr\^ace au lemme \ref{lmm 7}.

On a ainsi montré que $Q_n \in E_\tors \cap E(K_{r_n,s_n})$ pour tout $n$ assez grand. 
Le lemme \ref{lmm 6} montre que son ordre $M_n$ est premier à $p$. 
La place $w$ est donc non ramifiée dans $K(E[M_n])$ \cite[Chapter VII, exercise 7.9]{Silverman2}. 
Elle est également totalement ramifiée dans $K_{r_n,s_n}$ par le corollaire \ref{cor 2}. 
Par conséquent, $K_{r_n,s_n} \cap K(E[M_n])=K$. 
D'où \[Q_n\in E(K_{r_n,s_n}) \cap E(K(E[M_n]))=E(K).\]
En particulier, $\sigma_n Q_n=Q_n$ ou encore $\sigma_n^2 P_n=[2]\sigma_n P_n-P_n$. 
Par récurrence, on obtient $\sigma_n^j P_n=[j]\sigma_n P_n-[j-1]P_n$ pour tout $j\geq 2$. 
Comme $\sigma_n$ est d'ordre $p$, il s'ensuit en prenant $j=p$ que $[p]\sigma_n P_n= [p]P_n$ et donc $[p]Q_n=O_E$.
 L'ordre de $Q_n$ étant premier à $p$, on a $Q_n=\sigma_n P_n-P_n=O_E$.
Ainsi, $P_n\in E(K_{r_n,s_n}^{G_{r_n,s_n}})$ pour tout $i$ assez grand (on rappelle que $G_{r_n,s_n}$ est le groupe cyclique engendré par $\sigma_n$). 
Grâce à \eqref{eq 2}, on a soit $P_n\in E(K_{r_n-1,s_n})$, soit $P_n\in E(K_{r_n,s_n-1})$. 
La seconde possibilité contredit aussitôt la minimalité de $(r_n,s_n)$. 
On a donc $P_n\in E(K_{r_n-1,s_1})$. 
Mais alors, toujours par \eqref{eq 2}, cela implique que $r_n >s_n$, i.e. $r_n-1 \geq s_n$. 
Ceci contredit à nouveau la minimalité de $(r_n,s_n)$. 
En conclusion, $(r_n,s_n) \in\{(0,0),(1,0)\}$ pour tout $n$ assez grand, i.e. $P_n\in E(K_{1,0})$.
Comme $ \hat{h}_E(P_n)\to 0$, le théorème se déduit maintenant de celui de Northcott. 

\section{Appendice : Points de Ribet} \label{section 5} 
Dans cette section, nous d\'efinissons les points de Ribet d'une vari\'et\'e semi-ab\'elienne. 
Nous avons affirm\'e dans l'introduction que les points de Ribet de $\G_m^n \times A$ \'etaient ses points de torsion. 
Le but de cette section est de montrer ce fait. 

Soit $\G$ une variété semi-abélienne, extension d'une variété abélienne $A$ par $\G_m^n$, définie sur un corps de nombres $k$. 
Commençons par définir les points de Ribet de $\G$ (pour une construction plus détaillée dans le cas $n=1$, voir \cite[section 1]{JacquinotRibet}). 

Soit $P\in \G(\overline{k})$. 
Pour un entier $m\in\Z^*$, notons $[1/m]P\in \G(\overline{k})$ un point tel que $[m][1/m]P=P$ et $\G[m]$ le noyau du morphisme $[m] : \G(\overline{k}) \to \G(\overline{k})$. 
Avec nos précédentes notations, on a, pour $m\geq 1$ et $\alpha\in \overline{k}^*$, $\mu_m=\G_m[m]$ et $\alpha^{1/m}=[1/m]\alpha$.
Soit $K\subset \overline{k}$ un corps de nombres tel que $P\in\G(K)$ et remarquons que $[1/m]P$ est défini à un point de $\G[m]$ près. 
Par conséquent, $K(\G[m], [1/m]P)$ ne dépend pas du choix de $[1/m]P$, ce qui donne un sens à la définition ci-dessous.

\begin{defn} [points de Ribet] \label{defn 1}
On dira que $P\in\G(\overline{k})$ est un point de Ribet s'il existe un corps de nombres $K\subset\overline{k}$ tel que pour tout nombre premier $l$ assez grand, \[P\in \G(K) \; \text{et} \; K(\G[l], [1/l]P)=K(\G[l]).\]
Le groupe des points de Ribet de $\G$ sera noté $\G_\mathrm{Rib}$.  
\end{defn}

\begin{rqu} \label{rqu 6}
\rm{On peut aussi définir les points de Ribet en utilisant la géométrie d'Arakelov, voir par exemple \cite{ChambertLoir}.} 
\end{rqu}

Le groupe $\G_{\mathrm{Rib}}$ est non vide puisqu'il contient les points de torsion de $\G$.  

Dans le cas $n=1$, Jacquinot et Ribet ont montré que $\G_\mathrm{Rib}$ peut s'interpréter comme le groupe de division d'un certain groupe de $\G(\overline{k})$ \cite[Theorem 4.6]{JacquinotRibet}. 
J'ignore si cela reste valable dans ce cadre plus général, mais on peut cependant montrer facilement que $\G_\mathrm{Rib}$ est stable par extraction de racines.
Plus précisément : 

\begin{lmm} \label{lmm 8}
 Soient $P\in \G_\mathrm{Rib}$ et $m\in \Z^*$. 
 Alors $[1/m]P\in\G_\mathrm{Rib}$. 
 \end{lmm}
 
\begin{proof}
Posons $Q=[1/m]P$.
Comme $P\in\G_\mathrm{Rib}$, il existe un corps de nombres $K$ tel que $P\in\G(K)$ et $K(\G[l], [1/l]P)=K(\G[l])$ pour tout nombre premier $l$ assez grand.
Quitte à étendre $K$, on peut supposer que $Q\in \G(K)$. 

Soit $l>m$ un nombre premier. 
Comme $l$ et $m$ sont premiers entre eux, l'égalité de Bézout affirme que $K(\G[l], [1/l]Q)=K(\G[l], [m/l]Q)$. 
Comme $[m]Q=P$, on conclut de tout cela que $K(\G[l], [1/l]Q)=K(\G[l])$ pour tout nombre premier $l$ assez grand, i.e. $Q\in \G_\mathrm{Rib}$. 
Cela achève la preuve du lemme. 
\end{proof}

Nous pouvons maintenant montrer ce que l'on d\'esire. 

\begin{prop} \label{prop 2}
Soit $A$ une variété abélienne définie sur un corps de nombres $k$. 
Posons $\G=\G_m^n \times A$. Alors $\G_\mathrm{Rib}= \G_\tors$. 
\end{prop}

\begin{proof}
On a d\'ej\`a l'inclusion $\G_\tors \subset \G_\mathrm{Rib}$. 

Soit $Q=((\alpha_1,\dots,\alpha_n), P)\in \G_\mathrm{Rib}$ et montrons que $Q\in \G_\tors=\mu_\infty ^n \times A_\tors$.
Pour tout entier $m\geq 1$, on a $\G[m]=\mu_m^n \times A[m]$. 
Il existe donc un corps de nombres $K \subset \overline{k}$ tel que $Q\in \G(K)$ et tel que pour tout nombre premier $l$ assez grand, \[K(\mu_l, A[l], \alpha_1^{1/l},\dots, \alpha_n^{1/l}, [1/l]P)=K(\mu_l, A[l]). \]
Soit $i\in\{1,\dots, n\}$. 
Il est immédiat que $K(\mu_l, A[l], \alpha_i^{1/l}, [1/l]P)=K(\mu_l, A[l])$ pour tout nombre premier $l$ assez grand. 
Par conséquent, $(\alpha_i, P)\in (\G_m \times A)_\mathrm{Rib}$. 
En appliquant maintenant un théorème profond de Bertrand \cite[Theorem 4- Proposition 4]{Bertrand} avec $G=\G_m \times A$, on obtient $h(\alpha_i)=0$, i.e. $\alpha_i\in\mu_\infty$. 

La fin de la preuve est due à l'un des référés anonymes. 
Comme $(\alpha_i,P)\in (\G_m \times A)_\mathrm{Rib}$, on déduit de \cite[Theorem 4.6]{JacquinotRibet} et \cite[§4]{JacquinotRibet} l'existence d'un entier $n\geq 1$ tel que $[n]P=f(v)-f^t(v)$ pour certains morphismes $f, f^t : \mathrm{Ext}(A,\G_m) \to A$ et où $v$ désigne l'élément de $\mathrm{Ext}(A,\G_m)$ paramétrant l'extension de $A$ par $\G_m$. 
Ici, cette extension est scindée, donc $v=0$. 
Par conséquent, $[n]P=f(0)-f^t(0)=O_A$. 
D'où, $P\in A_\tors$ et la proposition s'ensuit. 
\end{proof}
 
 \bibliographystyle{plain}

\end{document}